\documentclass[11pt]{amsart}

\usepackage{amsfonts,amsmath,amssymb,amscd}

\usepackage{amsthm}
\usepackage{latexsym}
\usepackage[dvips]{graphicx}
\usepackage[dvips]{psfrag}
\usepackage[dvips]{color}
\usepackage{xypic}
\usepackage[abs]{overpic}
\usepackage{caption}
\usepackage[all]{xy}
\usepackage{pinlabel}

\newtheorem{theorem}{Theorem}[section]
\newtheorem{lemma}[theorem]{Lemma}

\newtheorem{prop}[theorem]{Proposition}

\theoremstyle{definition}
\newtheorem{definition}[theorem]{Definition}
\newtheorem{example}[theorem]{Example}

\newcommand{\Nbd}{\operatorname{Nbd}}
\newcommand{\cl}{\operatorname{cl}}

\begin{document}

\title[Braid group and leveling of a knot]{Braid group and leveling of a knot}

\author{Sangbum Cho}\thanks{
The first-named author is supported by the Basic Science Research Program through the National Research Foundation of Korea (NRF) funded by the Ministry of Education (NRF-201800000001768).}

\address{
Department of Mathematics Education \newline
\indent Hanyang University, Seoul 133-791, Korea}
\email{scho@hanyang.ac.kr}

\author{Yuya Koda}\thanks{The second-named author is supported in part by JSPS KAKENHI Grant Numbers 15H03620, 17K05254, 17H06463, and JST CREST Grant Number JPMJCR17J4.}

\address{
Department of Mathematics \newline
\indent Hiroshima University, 1-3-1 Kagamiyama, Higashi-Hiroshima, 739-8526, Japan}
\email{ykoda@hiroshima-u.ac.jp}

\author{Arim Seo}
\thanks{}
\address{Department of Mathematics Education  \newline
\indent Korea University, Seoul 136-701, Korea}
\email{arimseo@korea.ac.kr}


\date{\today}

\begin{abstract}
Any knot $K$ in genus-$1$ $1$-bridge position can be moved by isotopy to lie in a union of $n$ parallel tori tubed by $n-1$ tubes so that $K$ intersects each tube in two spanning arcs, which we call a leveling of the position.
The minimal $n$ for which this is possible is an invariant of the position, called the level number.
In this work, we describe the leveling by the braid group on two points in the torus, which yields a numerical invariant of the position, called the $(1, 1)$-length.
We show that the $(1, 1)$-length equals the level number.
We then find braid descriptions for $(1,1)$-positions of all $2$-bridge knots providing upper bounds for their level numbers, and also show that the $(-2, 3, 7)$-pretzel knot has level number two.
\end{abstract}

\maketitle

\section{Introduction}
A knot $K$ in the $3$-sphere is said to be in {\it genus-$1$ $1$-bridge position}, or simply in {\it $(1,1)$-position}, with respect to a standard torus $T$ if $T$ splits the $3$-sphere into two solid tori $V$ and $W$, and each of $K\cap V$ and $K\cap W$ is a single arc that is properly embedded in $V$ and $W$ respectively, and is parallel into $T$.
A knot is called a {\it $(1, 1)$-knot} if it can be isotoped to be in $(1, 1)$-position with respect to some standard torus.
In a collar $T \times [0,1]$ of $T$ in $V$, one may take $n$ parallel copies of the form $T\times \{t\}$ and tube each of the two consecutive copies to obtain a surface $F_n$ of genus $n$ in $T\times[0,1]$.
Then we say that a knot $K$ lies in {\it $n$-level position} with respect to $T$ if $K \subset F_n$ and $K$ meets each of the tubes in two arcs connecting the two ends of the tube.
Any knot in $(1,1)$-position with respect to $T$ can be moved into a knot in $n$-level position with respect to $T$ for some $n$ in a natural way.
We call it a {\it leveling} of the $(1,1)$-position.
The minimum such $n$ over all the possible level positions with respect to $T$ is an invariant of the $(1,1)$-position, called the {\it level number} of the position.
In particular, a knot has a $(1,1)$-position of level number one if and only if it is a torus knot.

The level number of a $(1, 1)$-position is equal to a Hempel-type complexity called the arc distance of the position.
Given a $(1, 1)$-position of $K$ with respect to $T$ as above,
the {\it arc complex} is the simplicial complex whose vertices are the isotopy classes of simple arcs in $T$ connecting the two points $K \cap T$.
Then the {\it arc distance} of the $(1, 1)$-position for $K$ is defined to be the minimum simplicial distance between the collection of vertices represented by arcs in $T$ from that are parallel to $K \cap V$ in $V$ and the analogous collection for $K \cap W$.
In fact, it was shown in \cite{CMcCS} that the level number equals the arc distance in more general setting, taking a genus-$g$ Heegaard surface $F$ and the $(g, 1)$-position of $K$ with respect to $F$ rather than the torus $T$ and the $(1, 1)$-position with respect to $T$.

On the other hand, any $(1,1)$-position of a knot can be described algebraically using the braid group on two points in the torus.
In fact, each of the $(1,1)$-positions of a knot corresponds to a collection of the words of the ``reduced'' braid group, that is the quotient of the braid group by its center.
Such a description of $(1, 1)$-positions was introduced and used in \cite{CMcC2} to compute the slope invariants of $(1, 1)$-tunnels of knots.
Using the words in the collection corresponding to a $(1, 1)$-position, we define a new numerical invariant, called the {\it $(1,1)$-length} of the position.
The main goal of this work is to show that the level number of a $(1, 1)$-position equals the $(1,1)$-length of the position (Theorem \ref{thm:equality}).

In Section \ref{sec:torus_leveling}, we introduce the level position of a $(1, 1)$-knot and leveling of a $(1, 1)$-position.
In Sections \ref{sec:braid_group} and \ref{sec:torus_words}, the reduced braid group and the $(1, 1)$-length are defined, and then our main theorem is proved in Section \ref{sec:characterization}.
In Section \ref{sec:example}, we carry out some explicit computations for $(1, 1)$-positions of $2$-bridge knots. Specifically, we find braid word descriptions for all the $(1,1)$-positions of each $2$-bridge knot, and provide upper bounds for their level numbers.
We also give a conjectural description of the $2$-bridge knots with level number $2$.
In the final section, we show that each of the $(1, 1)$-positions of $(-2, 3, 7)$-pretzel knot has level number two and obtain their braid descriptions.

Throughout the paper, we denote by $\Nbd(X)$ and  $\cl(X)$ a regular neighborhood of $X$ and the closure of $X$ respectively, where $X$ is a subspace of a polyhedral space.
The ambient spaces will be always clear from the context.
Finally, the authors are deeply grateful to Darryl McCullough for his valuable advice and comments.

\section{Leveling and level number}
\label{sec:torus_leveling}

Let $T$ be a standard torus in the $3$-sphere which splits the $3$-sphere into two solid tori $V$ and $W$.
We regard a collar of $T$ in $V$ as the product $T \times [0, 1]$ so that $T = T \times \{0\}$.
Choose the numbers $t_1, t_2, \ldots, t_n$ with $0 = t_1 < t_2 < \cdots <t_n = 1$, and denote the {\it level torus} $T \times \{t_j\}$ in the collar by $T_j$.
We construct a closed orientable surface $F_n$ of genus $n$ in the collar as follows.

First choose $n-1$ disks $D^1, D^2, \ldots, D^{n-1}$ in $T = T_1$ so that each $D^j$ is disjoint from $D^{j+1}$, and denote by $C_j$ the tube $\partial D^j \times [t_j, t_{j+1}]$ for $j \in \{1, 2, \ldots, n-1\}$.
Then, from the union $T_1 \cup C_1 \cup \cdots \cup T_{n-1} \cup C_{n-1} \cup T_n$, remove the interiors of $D^j \times \{t_j\}$ and $D^j \times \{t_{j+1}\}$ for $j \in \{1, 2, \ldots, n-1\}$ to get a closed surface $F_n$ of genus $n$.
We call such a surface $F_n$ is a surface {\it determined by} $T$.
The subsurfaces $F_n \cap T_1$ and $F_n \cap T_n$ are $1$-holed tori while $F_n \cap T_j$ for $j \in \{2, 3, \ldots, n-1\}$ $2$-holed tori. (See Figure \ref{fig:torus_level}.)

\medskip

\begin{center}
\includegraphics[width=8cm]{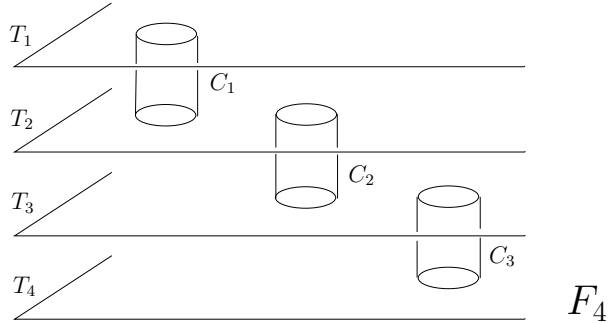}
\captionof{figure}{The genus-$4$ surface $F_4$.}
\label{fig:torus_level}
\end{center}

For an integer $n \geq 2$, a knot $K$ in the $3$-sphere is said to lie in {\it $n$-level position} with respect to $T$ if there exists a surface $F_n$ determined by $T$ such that $K \subset F_n$ and $K \cap C_j$ is a pair of  arcs connecting the two boundary circles of the tube $C_j$ for $j \in \{1, 2, \ldots, n-1\}$.
A knot $K$ lies in {\it $1$-level position} if $K$ lies in a standard torus in the $3$-sphere.
By definition, when a knot $K$ lies in $n$-level position with respect to $T$ for $n \geq 2$, then $F_n \cap T_j \cap K$ is a single arc properly embedded in $F_n \cap T_j$ if $j \in \{1, n\}$.
If $j \in \{2, 3,  \ldots, n-1\}$, then $F_n \cap T_j \cap K$ is a pair of disjoint arcs properly embedded in $F_n \cap T_j$ such that each arc connects the two boundary circles of  $F_n \cap T_j$.

\medskip

\begin{center}
\includegraphics[width=9.5cm, clip]{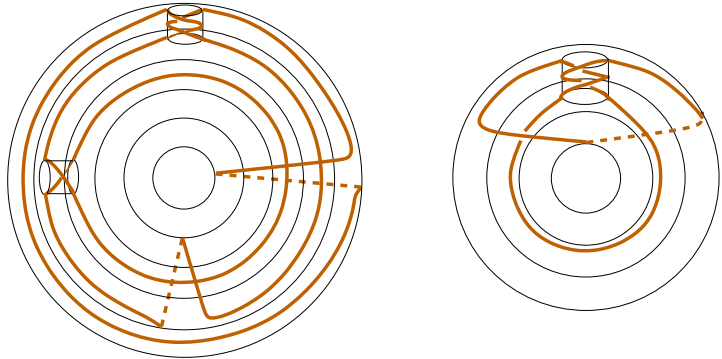}
\captionof{figure}{Examples of leveling. The second one is a 2-level position of the figure-eight knot.}
\label{fig:leveling_example}
\end{center}

\medskip

When a knot $K$ lies in $n$-level position with respect to $T$, $K$ can be put into a $(1, 1)$-position with respect to $T$ by pushing the arc $K \cap T$ into the solid torus $W$.
It is not hard to see that the converse is also true.
That is, given a knot $K$ in $(1, 1)$-position with respect to $T$, one can isotope $K$, keeping $K \cap V$ in $V$ and $K \cap W$ in $W$ at all times, so that $K$ lies in $n$-level position with respect to $T$ for some $n$.
This is also a direct consequence of Theorem 3.2 in \cite{CMcCS}.
Such a level position is called a {\it leveling} of the $(1, 1)$-position.
Now we define an integral invariant which measures the complexity of a $(1, 1)$-position of a knot.

\begin{definition}
The {\it level number} of a $(1, 1)$-position of $K$ is the minimum number of level tori over all the levelings of the $(1, 1)$-position.
The {\it level number} of a $(1, 1)$-knot $K$ is the minimum level number over all $(1, 1)$-positions of $K$.
\end{definition}

\section{The braid group and braid descriptions of $(1, 1)$-positions}
\label{sec:braid_group}

In this section, we briefly review a geometric interpretation of the braid group on two points in the torus, and braid descriptions of $(1, 1)$-positions, which were introduced in \cite{CMcC2}.
Let $T$ be a torus and let $N = T \times [0, 1]$.
Fix two points $b$ and $w$ in $T$, and denote by $b_t$, $w_t$ and $T_t$ the points $b \times \{t\}$ and $w \times \{t\}$ and the slice $T \times \{t\}$ respectively for each $t \in [0, 1]$.
Consider a pair of disjoint arcs properly embedded in $N$ such that each endpoint of the arcs is one of $b_0$, $b_1$, $w_0$ and $w_1$, and each of the arcs meets each slice $T_t$ transversely in a single point.
There is an obvious multiplication operation on the collection of such pairs defined by ``stacking'' two pairs.
Two such pairs are called {\it equivalent} if there is an isotopy $J_s$ of $N$ such that
\begin{enumerate}
\item $J_0 = id _N$,
\item $J_s |_{\partial N} = id_{\partial N}$ for $s \in [0, 1]$,
\item $J_s (T_t)= T_t$ for $t \in [0, 1]$ and $s \in [0, 1]$, and
\item $J_1$ sends one pair to the other pair.
\end{enumerate}
The equivalence classes are called {\it braids}, and the above multiplication operation induces a group structure on them.
We call the group of all braids under the induced multiplication the {\it $2$-braid group} on the torus.
A finite presentation of the $2$-braid group on the torus is well-known, which can be found in \cite{B}, \cite{Ta} or in \cite{CMcC2}.
We rewrite it as:
$$\langle~ m, l, s ~|~ msms=smsm, ~lsls=slsl, ~l^{-1}mlm^{-1}=s^2,
~lsms^{-1}=smsl~\rangle.$$

To describe the generators  $m$, $l$ and $s$ geometrically, first fix oriented meridian and longitude curves $m$ and $l$ in $T$.
(For convenience, we use the same symbols $m$ and $l$ for the generators and the curves.)
Denote by $b$ the point $m \cap l$, and fix a point $w$ disjoint from $m \cup l$.
Choose an arc $\alpha$ in $T$ connecting $b$ and $w$, and meeting $m \cup l$ only in $b$. There are four isotopy classes of such arcs in $T$, and we choose $\alpha$ such that $\alpha$ leaves $m$ in the direction of negative orientation of $l$ and leaves $l$ in the direction of positive orientation of $m$.
See Figure \ref{fig:torus}.

\begin{center}
\centering
\includegraphics[width=5cm, clip]{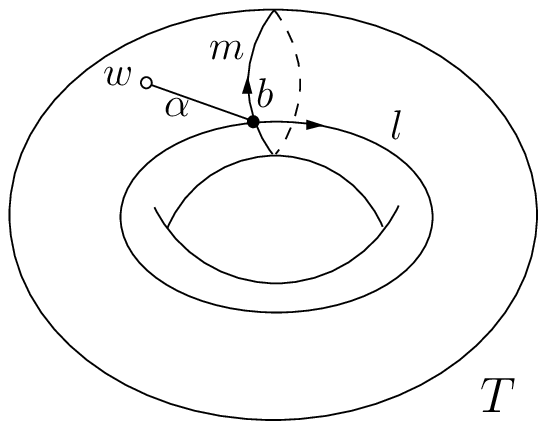}
\captionof{figure}{}
\label{fig:torus}
\end{center}

Then the generator $m$ (respectively $l$) can be represented by a pair of arcs in $N$ of which one arc connects the vertices $b_0$ and $b_1$ after sliding around $m$ (respectively $l$) in the direction of the orientation on $m$ (respectively $l$), while the other one is vertical as in Figure \ref{fig:braid}.
The element $s$ is represented by a pair of arcs which are half-twisted  as in Figure \ref{fig:braid2} (a).
Precisely, we choose a disk $D$ in $T$ such that $\alpha$ is properly embedded in $D$ and $D$ meets $m \cup l$ only in $b$.
Then the element $s$ is represented by a pair of spanning arcs of the cylinder $\partial D \times [0, 1]$ such that the arc connecting $b_0$ and $w_1$ over-crosses the other one once.
The relations can be verified directly.

\bigskip

\begin{center}
\includegraphics[width=10.5cm, clip]{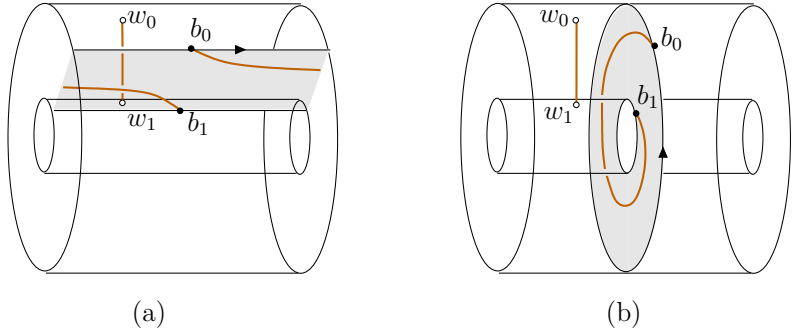}
\captionof{figure}{(a) a pair of arcs which represents $l$. (b) a pair of arcs which represents $m$.}
\label{fig:braid}
\end{center}

\bigskip

\begin{center}
\includegraphics[width=11cm, clip]{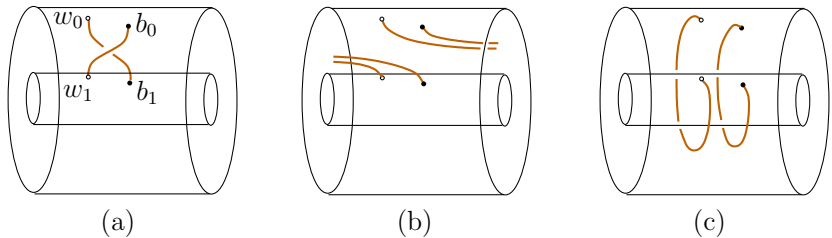}
\captionof{figure}{(a) a pair of arcs which represents $s$. (b) a pair of arcs which represents $lsls$. (c) a pair of arcs which represents $msms$.}
\label{fig:braid2}
\end{center}

Now we weaken condition $(2)$ for the isotopy $J_s$ to

\medskip
~~~$(2)'$ $J_1 |_{\partial N} = id_{\partial N}$.
\medskip

That is, we do not require that each $J_s$ is the identity on $\partial N$ for $s \in (0, 1)$.
We call the new equivalence classes of the pairs of arcs under this condition the {\it reduced braids}, and the group of all reduced braids the {\it reduced braid group} and denote it by $\mathcal B$.
The fundamental group $\pi_1(T) = \Bbb Z \times \Bbb Z$ of the torus $T$ can be considered as a subgroup of the $2$-braid group on the torus naturally.
That is, $\pi_1(T)$ is generated by $lsls$ and $msms$, which are represented by pairs of arcs described in Figure \ref{fig:braid2} (b) and (c).
One can verify that $\pi_1(T)$ is central in the $2$-braid group on the torus, and that the reduced braid group $\mathcal B$ is the quotient of the $2$-braid group by $\pi_1(T)$.
Thus by adding two relations $lsls = 1$ and $msms = 1$ to the above presentation, we obtain the following.

\begin{prop}
The reduced braid group $\mathcal B$ has the presentation
\begin{center}
$\langle ~m, l, s ~| ~msms = 1, ~lsls =1, ~l^{-1}mlm^{-1} = s^2 ~ \rangle$.
\end{center}
\end{prop}

Now suppose that the torus $T$ for the definition of the braid group is a standard torus in the $3$-sphere which splits the $3$-sphere into two solid tori $V$ and $W$, and that the product $N = T \times [0, 1]$ is a collar of $T$ in $V$ with $T = T \times \{0\}$.
The curves $m$ and $l$ in $T$ bound meridian disks of $V$ and $W$ respectively.
Then any word $\omega$ in $\mathcal B$ defines a knot $K(\omega)$ together with a $(1, 1)$-position of $K(\omega)$ with respect to $T$.
That is, the knot $K(\omega)$ is obtained by attaching two arcs $\alpha = \alpha \times \{0\}$ and $\alpha \times \{1\}$ to a pair of arcs representing $\omega$ and then by pushing the arc $\alpha$ slightly into $W$.
Then the knot $K(\omega)$ lies in $(1, 1)$-position with respect to $T$.

The knot $K(\omega)$ is well defined, indeed if two words $\omega$ and $\omega'$ are equivalent in $\mathcal B$ (i.e. represent the same braid) then the resulting knots $K(\omega)$ and $K(\omega')$ are $(1, 1)$-isotopic (that is, isotopic by an isotopy of the $3$-sphere preserving $T$ at all times).
The knot $K(\omega)$ is said to lie in {\it braid position} (with respect to $T$), and the word $\omega$ is called a {\it braid description} of the $(1, 1)$-position for $K(\omega)$ (after fixing the arc $\alpha$).
On the other hand, any knot $K$ that lies in  $(1, 1)$-position with respect to $T$ is $(1, 1)$-isotopic to a knot $K(\omega)$ for some word $\omega$ in $\mathcal B$.
In fact, we will see that any knot $K$ in level position with respect to $T$ can be repositioned to a knot $K(\omega)$ by isotopy keeping $K \cap W$ in $W$ at all times (Lemma \ref{lem:leveling_and_braid_position}).

By the construction, we observe that $K(\omega)$ is $(1, 1)$-isotopic to each of $K(\omega m)$, $K(l\omega)$, $K(\omega s)$ and $K(s\omega)$.
In general, if $\omega_L$ is a word containing only powers of $l$ and $s$, and $\omega_M$ of $m$ and $s$, then $K(\omega)$ is $(1, 1)$-isotopic to $K(\omega_L \omega \,\omega_M)$ for any word $\omega$ (of course, $\omega_L$ and $\omega_M$ are possibly empty words).

\begin{definition}
Let $\omega$ and $\omega'$ be words in the reduced braid group $\mathcal B$.
We say that $\omega$ is {\it $(1, 1)$-equivalent} to $\omega'$ if $\omega'$ is equivalent to $\omega_L \omega \,\omega_M$ in $\mathcal B$ for some words $\omega_L$ containing only powers of $l$ and $s$, and $\omega_M$ of $m$ and $s$.
We write $\omega \sim \omega'$ when $\omega$ is $(1, 1)$-equivalent to $\omega'$.
\end{definition}

The knots $K(\omega)$ and $K(\omega')$ are $(1, 1)$-isotopic if and only if $\omega$ and $\omega'$ are $(1, 1)$-equivalent (with respect to the torus $T$).

\section{Torus words in the reduced braid group}
\label{sec:torus_words}

In this section, we introduce the words of some special types in the reduced braid group $\mathcal B$, called the $(1, 1)$-words and then define the $(1, 1)$-length, a numerical invariant for $(1, 1)$-positions.

Consider the subspace $(\Bbb R \times \Bbb Z) \cup (\Bbb Z \times \Bbb R)$ of $\Bbb R^2$ as a $1$-dimensional simplicial complex.
That is, the vertices are the lattice points $\Bbb Z \times \Bbb Z$, and every edge is of length $1$ and is either vertical or horizontal.
Let $\widetilde{\gamma}$ be the line segment in $\Bbb R^2$ connecting the origin $(0, 0)$ and the point $(p, q)$ where $p$ and $q$ are nonzero, relatively prime integers.
There are many shortest paths in the complex $(\Bbb R \times \Bbb Z) \cup (\Bbb Z \times \Bbb R)$ from $(0, 0)$ to $(p, q)$.
Among them, we will choose a special one, denoted by $\gamma_{(p, q)}$, which is ``close'' to the segment $\widetilde{\gamma}$ in some sense.

Let $R_1, R_2, \ldots, R_k$ be the sequence of rectangles bounded by four edges of the complex $(\Bbb R \times \Bbb Z) \cup (\Bbb Z \times \Bbb R)$ such that each $R_j$ intersects $\widetilde{\gamma}$ in its interior as in Figure \ref{fig:rectangle}.
We assume that $R_1$ and $R_k$ have the vertices $(0, 0)$ and $(p, q)$ respectively.
Note that the length of the sequence of rectangles for the line segment from $(0, 0)$ to $(p, q)$ is $p+q-1$ (that is, $k = p+q-1$).

\bigskip

\begin{center}
\includegraphics[width=9cm, clip]{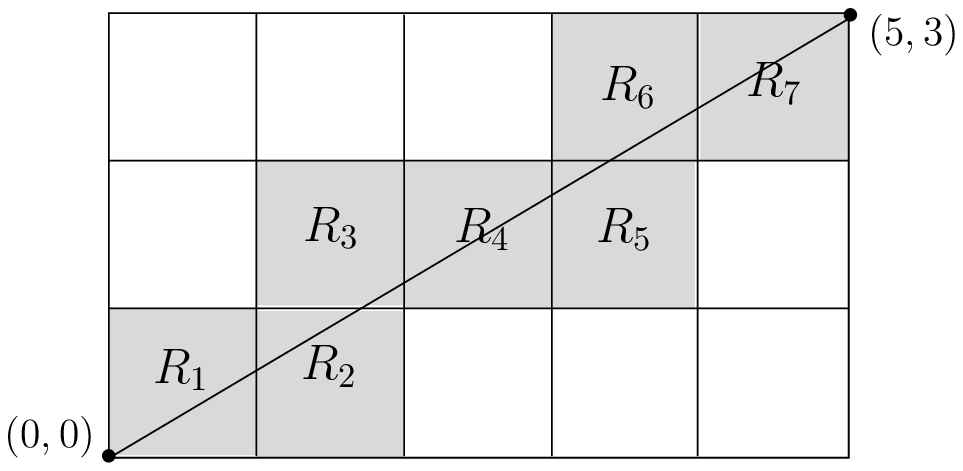}
\captionof{figure}{}
\label{fig:rectangle}
\end{center}

\begin{enumerate}
\item If $pq > 0$, choose all vertices of $R_j$ that are separated by $\widetilde{\gamma}$ from the lower right vertex of $R_j$ for each $j \in \{1, 2, \ldots, k\}$.
\item If $p > 0$ and $q < 0$, then choose vertices of $R_j$ as in the case of $pq > 0$ for each $j \in \{1, 2, \ldots, k-1\}$, and choose the two upper vertices of $R_k$.
\item If $p < 0$ and $q > 0$, then choose vertices of $R_j$ as in the case of $pq > 0$ for each $j \in \{2, 3, \ldots, k\}$, and choose the two upper vertices of $R_1$.
\end{enumerate}

Notice that some vertices may be chosen more than once.
Given the vertex $(p, q)$, we have the unique shortest path in $(\Bbb R \times \Bbb Z) \cup (\Bbb Z \times \Bbb R)$ from $(0, 0)$ to $(p, q)$ passing through the vertices chosen in (1), (2) or (3).
Denote by $\gamma_{(p, q)}$ this shortest path.
Note that the path $\gamma_{(p, q)}$ depends only on the choice of $(p, q)$.
We give a direction on each horizontal edge of $\gamma_{(p, q)}$ so that it is directed to the right if $p > 0$, and to the left if $p < 0$.
Similarly, each vertical edge is directed upward if $q > 0$, and downward if $q < 0$ (see Figure \ref{fig:torus_word}).
We call $\gamma_{(p, q)}$ the {\it directed path} from $(0, 0)$ to $(p, q)$.
We also regard the segments $\widetilde{\gamma}$ from $(0, 0)$ to $(0, \pm 1)$ and to $(\pm 1, 0)$ themselves as directed paths $\gamma_{(0, \pm 1)}$ and $\gamma_{(\pm 1, 0)}$ respectively.

\bigskip

\begin{center}
\includegraphics[width=12cm, clip]{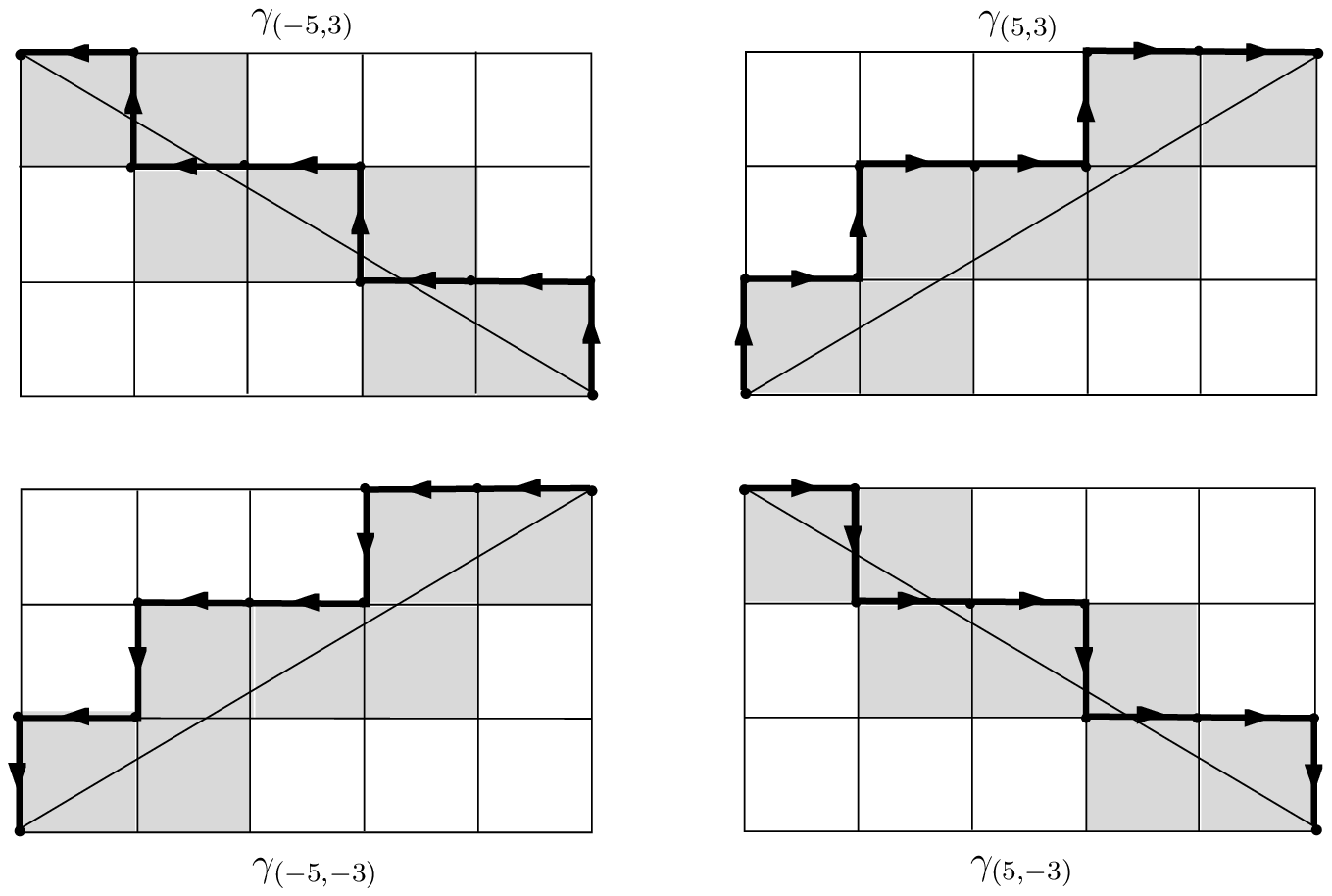}
\captionof{figure}{}
\label{fig:torus_word}
\end{center}

\begin{definition}
Let $p$ and $q$ be nonzero, relatively prime integers.
The {\it torus word} of type $(p, q)$ is a word in $\mathcal B$ determined by the directed path $\gamma_{(p, q)}$ as follows.
\begin{itemize}
\item Each horizontal edge gives a letter $l$ if it is directed to the right and a letter $l^{-1}$ if directed to the left. Similarly, each vertical edge gives a letter $m$ if it is directed upward and a letter $m^{-1}$ if directed downward.
    Then the torus word of type $(p, q)$ is the word read off along the directed path $\gamma_{(p, q)}$ from $(0, 0)$ to $(p, q)$.
\end{itemize}
Define the torus words of type $(1, 0)$, $(-1, 0)$, $(0, 1)$ and $(0, -1)$ to be $l$, $l ^{-1}$, $m$ and $m ^{-1}$ respectively.
\end{definition}

For example, the four directed paths $\gamma_{(5, 3)}$, $\gamma_{(-5, 3)}$, $\gamma_{(-5, -3)}$ and $\gamma_{(5, -3)}$ in Figure \ref{fig:torus_word} determine torus words of types  $(5, 3)$, $(-5, 3)$, $(-5, -3)$ and $(5, -3)$ as $mlml^2ml^2$, $ml^{-2} ml^{-2}ml^{-1}$, $l^{-2} m^{-1}l^{-2}m^{-1}l^{-1}m^{-1}$ and $lm^{-1}l^2m^{-1}l^2m^{-1}$ respectively.
By definition, any torus word does not contain $s$, and no torus word contains both $l$ and $l^{-1}$ (respectively both $m$ and $m^{-1}$) simultaneously.

Now we will describe a pair of arcs which represents a torus word in the reduced braid group $\mathcal B$ as follows.
Recall our basic setup for the definition of the reduced braid group $\mathcal B$ in Section \ref{sec:braid_group}.
Let $N = T \times [0, 1]$ be a regular neighborhood of the torus $T$.
Fix oriented meridian and longitude curves $m$ and $l$ in $T$ such that $m \cap l$ is the point $b$, and the point $w$ is disjoint from $m \cup l$.
Choose an arc $\alpha$ in $T$ connecting $b$ and $w$, and meeting $m \cup l$ only in $b$ such that $\alpha$ leaves $m$ in the direction of negative orientation of $l$ and leaves $l$ in the direction of positive orientation of $m$.
See Figure \ref{fig:torus}.
We also choose a disk $D$ in $T$ such that $\alpha$ is properly embedded in $D$ and $D$ meets $m \cup l$ only in $b$.
Under this setup the generators of the reduced braid group $\mathcal B$ were described geometrically in Section \ref{sec:braid_group}.
As usual, for any subspace $X$ of $T$, we denote by $X_t$ the parallel copy $X \times \{t\}$ in $T \times \{t\} \subset N$.

Consider a standard covering $\Bbb R^2$ of the torus $T$ in which the pre-images of $m$ and $l$ are vertical and horizontal lines $\Bbb Z \times \Bbb R$ and $\Bbb R \times \Bbb Z$ respectively. Then the pre-image of $m \cap l = \{b\}$ is the lattice $\Bbb Z \times \Bbb Z \subset \Bbb R^2$, and the pre-image of $w$ is disjoint from  $ (\Bbb Z \times \Bbb R) \cup (\Bbb R \times \Bbb Z)$.
We will assume that the pre-image of the disk $D$ intersects each rectangle bounded by four edges of $ (\Bbb Z \times \Bbb R) \cup (\Bbb R \times \Bbb Z)$ only in the lower right vertex.
That is, we give the orientation of $\Bbb R \times \Bbb Z$ from the left to the right, and the orientation of $\Bbb Z \times \Bbb R$ from the downside to the upside.

Recall that $\widetilde{\gamma}$ is the line segment in $\Bbb R^2$ connecting the origin $(0, 0)$ and the point $(p, q)$ where $p$ and $q$ are relatively prime integers, and $\gamma_{(p, q)}$ is the directed path in $\Bbb R^2$.
The path $\gamma_{(p, q)}$ is isotopic in $\Bbb R^2$ to the segment $\widetilde{\gamma}$ fixing the two endpoints (or is equal to $\widetilde{\gamma}$ if $(p, q)$ is $(0, \pm 1)$ or $(\pm 1, 0)$).
Furthermore, we can place the disk $D$ in so that $\gamma_{(p, q)}$ is isotopic to $\widetilde{\gamma}$ without crossing the pre-image of the disk $D$ except the two endpoints $(0, 0)$ and $(p, q)$.
That is, the arc $\widetilde{\gamma}$ divides each rectangle $R_j$ into two portions.
Then we say that the disk $D$ lies in {\it good position} (with respect to $\widetilde{\gamma}$) if $D$ lies in the intersection of the images of the portions having the lower right vertex of $R_j$.
(If $pq < 0$, both two portions of the first rectangle $R_1$ or the last rectangle $R_k$ have the lower right vertex of the rectangle.
In this case, we just choose the lower portion of the rectangle divided by $\widetilde{\gamma}$.)

We extend the covering $\Bbb R^2$ of $T$ to the covering $\Bbb R^2 \times [0, 1]$ of $N = T \times [0, 1]$ naturally.
Considering the directed path $\gamma_{(p, q)}$ in $\Bbb R^2$ as an embedding $\gamma_{(p, q)} = \gamma_{(p, q)} (t)$, for $t \in [0, 1]$, we define the arc $\overline{\gamma}_{(p, q)} = \{ (\gamma_{(p, q)}(t), t) | t \in [0, 1]\}$ in $\Bbb R^2 \times [0, 1]$.
Let $\lambda$ be the image of $\overline{\gamma}_{(p, q)}$ in $N = T \times [0, 1]$.
Then the arc $\lambda$ is properly embedded in $\cl (N - D \times [0, 1])$ and connects $b_0$ to $b_1$.
Furthermore, the arc $\lambda$ together with the vertical arc $\mu = w \times [0, 1]$ forms a pair of arcs which represents the torus word of type $(p, q)$.

On the other hand, considering $\widetilde{\gamma}$ as an embedding $\widetilde{\gamma} = \widetilde{\gamma}(t)$, for $t \in [0, 1]$, the arc $\overline{\gamma}_{(p, q)}$ is isotopic to the arc $\{ (\widetilde{\gamma}(t), t) | t \in [0, 1]\}$ without crossing the pre-image of $D \times [0, 1]$ except the two endpoints $(0, 0) \times \{0\}$ and $(p, q) \times \{1\}$.
Let $\lambda'$ be the image of $\{ (\widetilde{\gamma}(t), t) | t \in [0, 1]\}$ in $N = T \times [0, 1]$.
Then $\lambda'$ is an arc connecting the points $b_0$ to $b_1$ which is isotopic to $\lambda$ in $N$ without crossing $D \times [0, 1]$ except $b_0$ and $b_1$.
Note that the arc $\lambda'$ projects into a $(p, q)$-torus knot in $T$ which is the image of the line segment $\widetilde{\gamma}$.
We summarize the above observation as follows.

\begin{lemma}
Under the setup with $N = T \times [0, 1]$ and its covering $\Bbb R^2 \times [0, 1]$ in the above, a word $\omega$ in $\mathcal B$ is a torus word of type $(p, q)$ if and only if $\omega$ is represented by a pair of arcs $\{\lambda', \mu\}$ in $N$ satisfying:
\begin{itemize}
\item $\mu = \{w\} \times [0, 1]$,
\item $\lambda'$ projects into the $(p, q)$-torus knot in $T$ that is the image of the line segment $\widetilde{\gamma}$ in $\Bbb R^2$ joining $(0, 0)$ and $(p, q)$, and
\item $D$ lies in good position, that is, $\lambda'$ is isotopic to $\lambda$ in $N$ without crossing $D \times [0, 1]$ except the endpoints $b_0$ and $b_1$.
\end{itemize}
\label{lem:torus word}
\end{lemma}

\begin{definition}
A {\it $(1, 1)$-word} in the reduced braid group $\mathcal B$ is the empty word or a word of the form
\begin{center}
$\omega_1 s^{n_1} \omega_2 s^{n_2} \cdots \omega_{k-1} s^{n_{k-1}} \omega_k$,
\end{center}
where, $\omega_j$ is a torus word and $n_j$ is an integer (possibly $0$).
\end{definition}

Not all words are $(1, 1)$-words.
For example a nonzero power of $s$ is not a $(1, 1)$-word.
But by definition any word $\omega$ in $\mathcal B$ is equivalent to some $(1, 1)$-word.
Given a nonempty $(1, 1)$-word $\omega'$, define $|\omega'|$ to be the minimum number $k$ for all expressions $\omega' = \omega_1 s^{n_1} \omega_2 s^{n_2} \cdots \omega_{k-1} s^{n_{k-1}} \omega_k$ as above.
For the empty word $1$, define $|1|$ to be $1$.

\begin{definition}
Let $\omega$ be a word in $\mathcal B$.
The {\it $(1, 1)$-length} of $\omega$, denoted by $\| \omega \|$, is the minimal number $|\omega'|$ over all the $(1, 1)$-words $\omega'$  which are $(1, 1)$-equivalent to $\omega$.
The {\it $(1, 1)$-length} of a $(1, 1)$-position is the $(1, 1)$-length over all braid description of the $(1, 1)$-position.
The {\it $(1, 1)$-length} of a $(1, 1)$-knot $K$ is the $(1, 1)$-length over all $(1, 1)$-positions of $K$.
\end{definition}

Every torus word has $(1, 1)$-length $1$, and we have  $\| m \| = \| l\| = \|s\| = \|m s m s\| = \|l s l s\| = \|l^{-1}m l m^{-1}\| = 1$.
Also $\|m l m l\| = 1$ since
$$mlml \sim ls^2mml \sim m^2l$$
and $m^2l$ is a torus word.
Observe that $ml^{\pm 1}$ and $l^{\pm 1}m^{-1}$ are torus words while $l^{\pm 1} m$ and $m^{-1}l^{\pm 1}$ are not.
But all of them have $(1, 1)$-length $1$.
For example,
$$m ^{-1} l ^{-1} \sim l ^{-1} m ^{-1} l ^{-1} m ^{-2}$$
and $l ^{-1} m ^{-1} l ^{-1} m ^{-2}$ is a torus word.

\section{Algebraic Description of Leveling}
\label{sec:characterization}

In this section, we will show that the level number is actually equal to the $(1, 1)$-length.
We will use the same notations for the definitions of  leveling and the braid description in the previous sections.

\begin{lemma}
Let $K$ be a knot which lies in $(1, 1)$-position with respect to a standard torus $T$ in the $3$-sphere.
If the $(1, 1)$-position has a leveling of an $n$-level position with respect to $T$, then there is a braid description $\omega$ of the $(1, 1)$-position such that $|w|$ is at most $n$.
Conversely, if $\omega$ is a braid description of the $(1, 1)$-position with $(1, 1)$-length $n$, then there is a leveling of the $(1, 1)$-position with $n$ level tori.
\label{lem:leveling_and_braid_position}
\end{lemma}

\begin{proof}
To begin with, we recall the previous setup for the braid description of a $(1, 1)$-position.
The torus $T$ splits the $3$-sphere into two solid tori $V$ and $W$, and we fix the oriented meridian and longitude curves $m$ and $l$ of $T$, meeting in a single point $b$ and bounding disks in $V$ and $W$ respectively.
We choose a point $w$ in $T$ disjoint from $m \cup l$, and an arc $\alpha$ in $T$ connecting $b$ and $w$, and meeting $m \cup l$ only in $b$ such that $\alpha$ leaves $m$ in the direction of negative orientation of $l$ and leaves $l$ in the direction of positive orientation of $m$ as in Figure \ref{fig:torus}.
We also choose a disk $D$ in $T$ such that $\alpha$ is properly embedded in $D$ and $D$ meets $m \cup l$ only in $b$.
In the standard covering $\Bbb R^2$ of the torus $T$, the pre-images of $m$ and $l$ are vertical and horizontal lines $\Bbb Z \times \Bbb R$ and $\Bbb R \times \Bbb Z$ respectively.
We may assume that the pre-image in $\Bbb R^2$ of the disk $D$ in $T$ intersects each rectangle bounded by four edges of $ (\Bbb Z \times \Bbb R) \cup (\Bbb R \times \Bbb Z)$ only in the lower right vertex.
We also consider the standard covering $\Bbb R^2 \times I$ of $T \times I$.
As usual, for any subspace $X$ of $T$ or $\mathbb R^2$, we denote by $X_t$ the parallel copy $X \times \{t\}$ in $T \times \{t\} \subset T \times I$ or in $\Bbb R^2 \times \{t\} \subset \Bbb R^2 \times I$.

\medskip

Now suppose that a given $(1, 1)$-position for a knot $K$ with respect to $T$ admits a leveling of $n$-level position with respect to $T$.
We will assume $n \geq 3$. (The case of $n = 1, 2$ is similar but simpler.)
Then $K$ lies in a genus-$n$ surface $F_n$ described as follows.
We first regard a collar of $T$ in $V$ as the product $T \times [1, n]$ rather than $T \times [0, 1]$ for convenience, and let $T = T \times \{1\}$.
There are $n-1$ disks $D^1, D^2, \ldots, D^{n-1}$ in $T = T_1$ such that each $D^j$ is disjoint from $D^{j+1}$.
We denote by $C_j$ the tube $\partial D^j \times [j, j+1]$ for $j \in \{1, 2, \ldots, n-1\}$.
Then the surface $F_n$ is obtained from the union $T_1 \cup C_1 \cup \cdots \cup T_{n-1} \cup C_{n-1} \cup T_n$ by removing the interiors of $D^j \times \{j\}$ and $D^j \times \{j+1\}$ for $j \in \{1, 2, \ldots, n-1\}$.
The intersection $K \cap C_j$ is a pair of spanning arcs connecting the two boundary circles of the tube $C_j$ for $j \in \{1, 2, \ldots, n-1\}$.
By an isotopy, we may assume the disk $D^1$ in $T_1$ is identified with the disk $D$ defined in the previous paragraph, and further the two arcs $K \cap D \times [1, 3/2]$ are exactly vertical ones, $b \times [1, 3/2]$ and $w \times [1, 3/2]$ (see Figure \ref{fig:first_level} (a)).

\begin{center}
\includegraphics[width=12.7cm, clip]{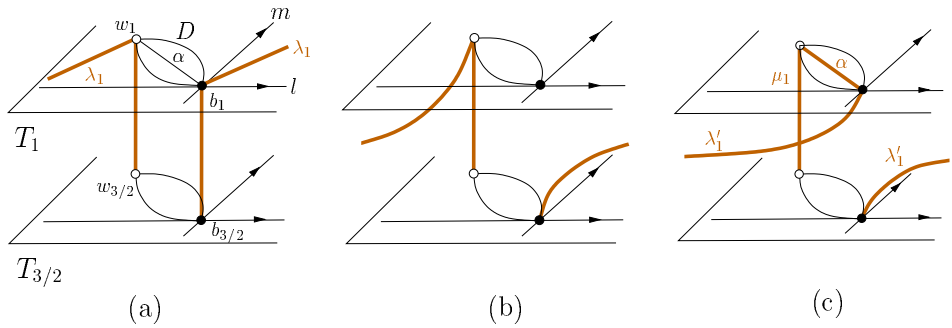}
\captionof{figure}{The reposition of $K \cap T \times [1, 3/2]$.}
\label{fig:first_level}
\end{center}

The union of the arc $K \cap T_1$ and the arc $\alpha$ is isotopic to a $(p_1, q_1)$-torus knot in $T_1$ for some relatively prime integers $p_1$ an $q_1$.
By an isotopy, we further assume that the disk $D$ lies in good position with respect to the line segment $\widetilde \gamma_1$ in $\mathbb R^2$ joining the points $(0, 0)$ and $(p_1, q_1)$ as in Lemma \ref{lem:torus word}.

We will reposition the arc $K \cap T \times [1, 3/2]$ by an isotopy fixing the portion of $K$ outside $T \times [1, 3/2]$ as follows.
We consider the arc $K \cap T_1$ as an embedding $\lambda_1 = \lambda_1(t)$ in $T_1$, for $t \in [1, 3/2]$, connecting $w_1$ to $b_1$.
Then the arc $K \cap T \times [1, 3/2]$ can be moved by isotopy of $K$ to the arc that is the union of the two arcs $\{(\lambda_1(t), t)|t \in [1, 3/2]\}$ and $w \times [1, 3/2]$, fixing the outside portion of $K$ (see Figure \ref{fig:first_level} (b)).
Next, by an isotopy of the arc $\{(\lambda_1(t), t)|t \in [1, 3/2]\}$, we move its top endpoint $w_1$ to the point $b_1$ along $\partial D$ so that the resulting arc, denoted by $\lambda'_1$, projects into the $(p_1, q_1)$-torus knot in $T_1$ which is the image of the line segment $\widetilde \gamma_1$ in $\mathbb R^2$.
Finally, adding the arc $\alpha$ to the unions of arcs $\lambda'_1$ and $w \times [1, 3/2]$, we have a repositioning of $K$ inside $T \times [1, 3/2]$, fixing the outside portion of $K$ (see Figure \ref{fig:first_level} (c)).
Denoted by $\mu_1$ the arc $w \times [1, 3/2]$, the pair of arcs $\{\lambda'_1, \mu_1\}$ represents the torus word, say $\omega_1$, of type $(p_1, q_1)$ by Lemma \ref{lem:torus word}.
We denote by $K$ again the knot after the repositioning of $K$.

The next step is to reposition the arcs of $K \cap T \times [2, 5/2]$ in a similar way, to obtain a torus word corresponding to the arcs.
We denote by $\lambda_2$ and $\mu$ the two arcs of $K \cap T_2$ where $\lambda_2$ has the endpoints $b_2$ in $\partial D_2$ and $b'_2$ in $\partial D^2_2$ while $\mu$ has $w_2$ in $\partial D_2$ and $w'_2$ in $\partial D^2_2$.
After an isotopy, we may assume the two arcs $K \cap D^2_2 \times [2, 5/2]$ are exactly vertical ones, $b'_2 \times [2, 5/2]$ and $w'_2 \times [2, 5/2]$ (see Figure \ref{fig:second_level} (a)).

The arc $\alpha_2$ is an arc properly embedded in $D_2$ with endpoints $b_2$ and $w_2$.
We choose an arc, denoted by $\beta$, properly embedded in $D^2_2$ with endpoints $b'_2$ and $w'_2$.
Then the union of $\lambda_2$, $\alpha_2$, $\mu$ and $\beta$ is isotopic to a $(p_2, q_2)$-torus knot in $T_2$ for some relatively prime integers $p_2$ and $q_2$.
By an isotopy again, we assume that the disk $D_2 = D \times \{2\}$ also lies in good position with respect to the line segment $\widetilde \gamma_2$ in $\mathbb R^2$ joining the points $(0, 0)$ and $(p_2, q_2)$ as in Lemma \ref{lem:torus word}.

\begin{center}
\includegraphics[width=12.7cm, clip]{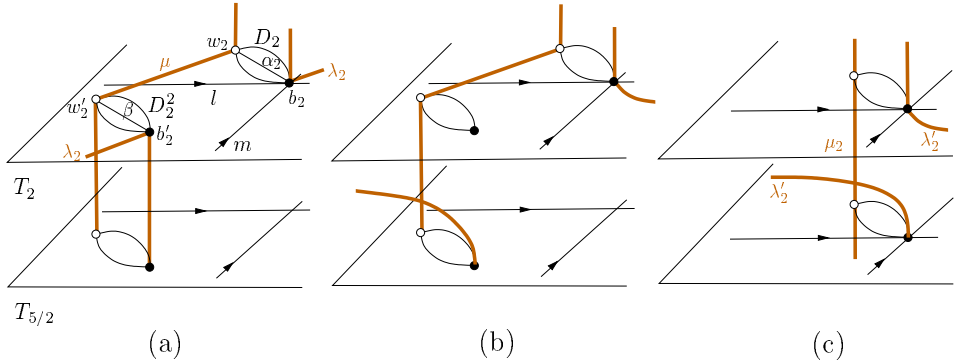}
\captionof{figure}{The reposition of $K \cap T \times [2, 5/2]$.}
\label{fig:second_level}
\end{center}

We regard the arc $\lambda_2$ as an embedding $\lambda_2 = \lambda_2(t)$ in $T_2$, for $t \in [2, 5/2]$, connecting $b_2$ to $b'_2$.
Then the arc $\lambda_2 \cup b'_2 \times [2, 5/2]$ can be moved by isotopy to the arc $\{(\lambda_1(t), t)|t \in [2, 5/2]\}$, fixing the remaining part of $K$ (see Figure \ref{fig:second_level} (b)).
Next, we isotope a small neighborhood of $T \times [2, n]$ so that the tube $C_2$ is identified with the tube $\partial D \times [2, 3]$, and
\begin{itemize}
\item the arc $\mu \cup w'_2 \times [2, 5/2]$ is moved to the vertical arc $w_2 \times [2, 5/2]$, and
\item the arc $\{(\lambda_2(t), t)|t \in [2, 5/2]\}$ is moved to the arc, denoted by $\lambda'_2$, joining $b_2$ and $b_{5/2}$, and it projects into the $(p_2, q_2)$-torus knot in $T_2$ which is the image of the line segment $\widetilde \gamma_2$ in $\Bbb R^2$.
\end{itemize}
See Figure \ref{fig:second_level} (c).
Denoted by $\mu_2$ the arc $w_2 \times [2, 5/2]$, the pair of arcs $\{\lambda'_2, \mu_2\}$ represents the torus word, say $\omega_2$, of type $(p_2, q_2)$, by Lemma \ref{lem:torus word} again.
We denote by $K$ again the knot after the repositioning of $K$.

We continue the process of the isotopy to obtain torus words $\omega_j$ consecutively for each $j \in \{2, 3, \ldots, n-1\}$.
For the final step, let us consider the collar $T \times [n, n+1]$ of $T_n$ in $\cl(V-T\times [1, n])$.
The disk $D^{n-1}_n$ was identified with the disk $D_{n-1}$, and hence the arc $K \cap T_n$ has endpoints $b_n$ and $w_n$ in $\partial D_n$.
The union of the arc $\lambda_n = K \cap T_n$ and the arc $\alpha_n$ is isotopic to the $(p_n, q_n)$-torus knot in $T_n$ for some relatively prime integers $p_n$ an $q_n$.
By an isotopy again, we assume that the disk $D$ also lies in good position with respect to the line segment $\widetilde \gamma_n$ in $\mathbb R^2$ joining the points $(0, 0)$ and $(p_n, q_n)$ as in Lemma \ref{lem:torus word}.
We reposition the arc $K \cap T_n$ to an arc $K \cap T \times [n, n+1]$ by isotopy fixing the portion of $K$ outside $T_n$ as in the above.
We consider the arc $\lambda_n$ as an embedding $\lambda_n = \lambda_n(t)$ in $T_n$, for $t \in [n, n+1]$, connecting $b_n$ to $w_n$.
See Figure \ref{fig:last_level} (a).

\begin{center}
\includegraphics[width=12.7cm, clip]{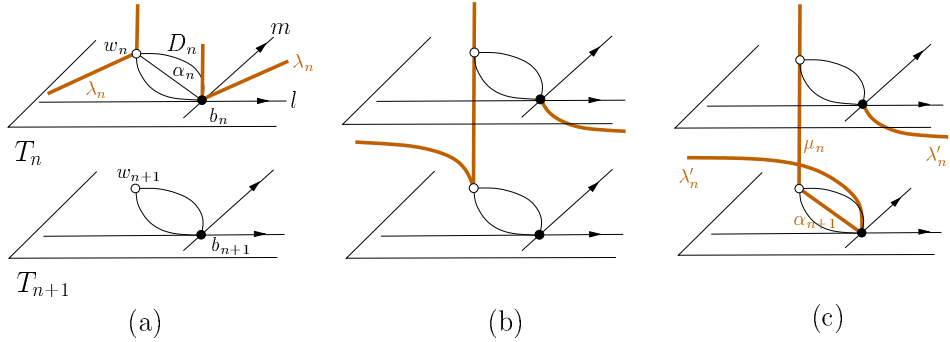}
\captionof{figure}{The reposition of $K \cap T \times [n, n+1]$.}
\label{fig:last_level}
\end{center}

Then the arc $\lambda_n$ can be moved by isotopy of $K$ to the arc that is the union of the two arcs $\{(\lambda_n(t), t)|t \in [n, n+1]\}$ and $w_n \times [n, n+1]$, fixing the outside portion of $K$ (see Figure \ref{fig:last_level} (b)).
Next, by an isotopy of the arc $\{(\lambda_n(t), t)|t \in [n, n+1]\}$, we move its bottom endpoint $w_{n+1}$ to the point $b_{n+1}$ along $\partial D_{n+1}$ so that the resulting arc, denoted by $\lambda'_n$, projects into the $(p_n, q_n)$-torus knot in $T_n$ which is the image of the line segment $\widetilde \gamma_n$ in $\mathbb R^2$.
Finally, adding the arc $\alpha_{n+1}$ to the unions of arcs $\lambda'_n$ and $w \times [n, n+1]$, we have a repositioning of $K$ inside $T \times [n, n+1]$ (see Figure \ref{fig:last_level} (c)).
Denoted by $\mu_n$ the arc $w \times [n, n+1]$, the pair of arcs $\{\lambda'_n, \mu_n\}$ represents the torus word, say $\omega_n$, of type $(p_n, q_n)$, by Lemma \ref{lem:torus word}.
Finally, we obtain a braid description
\begin{center}
$\omega = \omega_1 s^{k_1} \omega_2 s^{k_2} \cdots \omega_{n-1} s^{k_{n-1}} \omega_n$,
\end{center}
of the $(1, 1)$-position, where $\omega_j$ is the torus word of type $(p_j, q_j)$ for each $j \in \{1, 2, \ldots, n\}$ and the integer $k_j$ is determined by the half twists of the portion $K \cap \partial D \times [j+\frac{1}{2}, j+1]$ for each $j \in \{1, 2, \ldots, n-1\}$.

\medskip

Conversely, let $\omega$ be a braid description of the $(1, 1)$-position with $(1, 1)$-length $n$.
Then $\omega$ is written as $\omega_1 s^{k_1} \omega_2 s^{k_2} \cdots \omega_{n-1} s^{k_{n-1}} \omega_n$ in the above.
We may assume that the word $\omega$ is represented by a pair of arcs in a collar $T \times [1, n+1]$ of $T = T_1$ in $V$.
Then the knot $K$ is the union of the two arcs in the pair with the arcs $\alpha_1$ in $T_1$ and $\alpha_{n+1}$ in $T_{n+1}$.
By the reverse process of the above argument, we obtain a surface $F_n$ in $T \times [1, n+1]$ in which $K$ is positioned in $n$-level position with respect to $T$.
\end{proof}

The following is our main result, a direct consequence of Lemma \ref{lem:leveling_and_braid_position}.

\begin{theorem}
The level number of a $(1,1)$-position equals the $(1, 1)$-length of the position.
The level number of a $(1, 1)$-knot equals the $(1, 1)$-length of the knot.
\label{thm:equality}
\end{theorem}

\section{Braid descriptions of $2$-bridge Knots}
\label{sec:example}

In this section, we will discuss the braid descriptions and level numbers of the $(1, 1)$-positions for  $2$-bridge knots.
The braid descriptions for $2$-bridge knots were also introduced in \cite{CMcC2}, which we will refine in detail.
A knot $K$ in the $3$-sphere is called a {\it $2$-bridge knot} if there is a $2$-sphere $S$ such that $S$ splits the $3$-sphere into two $3$-balls $B_1$ and $B_2$, and each of $K \cap B_i$ consists of two disjoint properly embedded arc in $B_i$ parallel into $S$.
We call such a decomposition for a $2$-bridge knot $K$ a {\it $2$-bridge position} for $K$.
It is well-known that the $2$-bridge position for any $2$-bridge knot is unique up to equivalence (see \cite{Scb}).
The $(1, 1)$-positions of $2$-bridge knots can be described by their tunnels as we will see in the following.

A {\it tunnel} for a knot $K$ in the $3$-sphere is a simeple arc $\tau$ such that $\tau \cap K = \partial \tau$ and the exterior of $K \cup \tau$ is the genus 2 handlebody.
A knot which admits a tunnel is called a {\it tunnel number one knot}.
Any $(1, 1)$-knot $K$ is tunnel number one.
In fact, once we have a $(1, 1)$-position of $K$ with respect to a standard torus which splits the $3$-sphere into two solid tori $V$ and $W$, it is not hard to find tunnels $\mu$ in $V$ and $\mu'$ in $W$ for $K$.
The tunnel $\mu$ is an arc in $V$ whose endpoints lie in the interior of the arc $\alpha = V \cap K$ such that $V$ is a regular neighborhood of $\mu \cup \alpha$ (after pushing the endpoints of $\alpha$ slightly into the interior of $V$).
The tunnel $\mu'$ in $W$ is described in a similar way.
See Figure \ref{fig:decomp} (a).
Such tunnels $\mu$ and $\mu'$ are called {\it $(1, 1)$-tunnels} for $K$ (with respect to the standard torus of the $(1, 1)$-position).

\bigskip

\begin{center}
\includegraphics[width=9.5cm, clip]{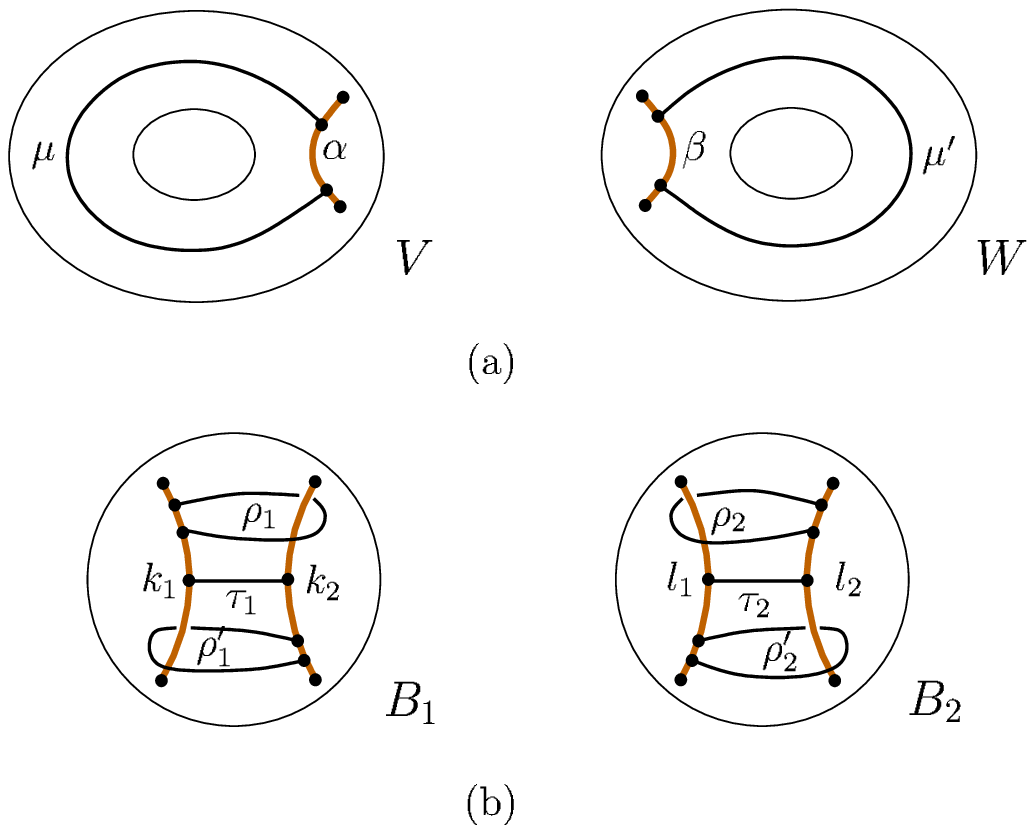}
\captionof{figure}{}
\label{fig:decomp}
\end{center}

Now let $K$ be a $2$-bridge knot and let $S$ be a $2$-sphere which splits the $3$-sphere into two $3$-balls $B_1$ and $B_2$, and each of $K \cap B_i$ consists of two disjoint properly embedded arc in $B_i$ parallel into $S$.
We denoted by $k_1$ and $k_2$ the two arcs $K \cap B_1$ in $B_1$ and by $l_1$ and $l_2$ the two arcs $K \cap B_2$ in $B_2$ as in Figure \ref{fig:decomp} (b).
In the figure, the arc $\tau_1$ in $B_1$ is an arc whose endpoints lie in the interiors of $k_1$ and $k_2$ such that a regular neighborhood of $k_1 \cup \tau_1 \cup k_2$ is $B_1$ (after pushing the endpoints of $k_1$ and $k_2$ slightly into the interior of $B_1$).
The arc $\rho_1$ in $B_1$ is an arc whose endpoints lie in the interior of $k_1$ such that $\cl(B_1 - \Nbd(k_2))$ is a regular neighborhood of the union of $\rho_1$ and the sub-arc of $k_1$ having the same endpoints of $\rho_1$.
The other arcs $\rho'_1$, $\tau_2$, $\rho_2$ and $\rho'_2$ are described similarly.
It is immediate that the six arcs $\tau_1$, $\tau_2$, $\rho_1$, $\rho_2$, $\rho'_1$ and $\rho'_2$ are all $(1, 1)$-tunnels for the $2$-bridge knot $K$.
Furthermore, it was shown in \cite{Ko} in that every tunnel of a $2$-bridge knot $K$ is isotopic to one of the six tunnels.

On the other hand, there is a natural way to obtain a $(1, 1)$-position from a $2$-bridge position of a $2$-bridge knot.
First, we choose an arc, say $k_1$, among the four arcs of $K \cap B_i$ in a  $2$-bridge position of $K$ as in Figure \ref{fig:decomp} (b).
Let $V = \cl(B_1 - \Nbd(k_1))$ and $W = B_2 \cup \Nbd(k_1)$.
Then we have a $(1, 1)$-position of $K$ with respect to the torus $\partial V = \partial W$.
There are three more $(1, 1)$-positions obtained from the $2$-bridge position depending on the choice of other arcs $k_2$, $l_1$ and $l_2$ in a similar way.
We call such a $(1, 1)$-position a {\it stabilization} of the $2$-bridge position.
Since the $2$-bridge position of a $2$-bridge knot $K$ is unique up to equivalence, there are four equivalence classes of stabilizations depending on the choice of one of the four arcs.
Furthermore, from Kobayashi and Saeki \cite{KoS} (and later from \cite{CK}), it was shown that any $(1, 1)$-position of a $2$-bridge knot can be obtained in this way.
We observe that, in the $(1, 1)$-position from the choice of arc $k_1$, the $(1, 1)$-tunnels $\mu$ in $V$ and $\mu'$ in $W$ in Figure \ref{fig:decomp} (a) are identified with $\rho'_1$ and $\tau_2$ in Figure \ref{fig:decomp} (b) respectively.
For the choice of other arcs $k_2$, $l_1$ and $l_2$, the tunnel $\mu$ in $V$ is identified with $\rho_1$, $\rho_2$ and $\rho_2'$ respectively, and the tunnel $\mu'$ in $W$ is identified with $\tau_2$, $\tau_1$ and $\tau_1$ respectively.
For each $\rho \in \{\rho_1, \rho'_1, \rho_2, \rho'_2\}$, the $(1, 1)$-position having the $(1, 1)$-tunnel $\rho$ will be called the {\it $\rho$-position} simply.
Of course, some of the four $(1, 1)$-positions are possibly identified by isotopy.

\medskip

There is a well-known classification of 2-bridge knots $K$ by rational parameters $p/q$.
Here $p$ is an odd positive integer, and we may assume that $q$ is even and $0 < |q| < p$.
We call $K = K_{p/q}$ the  $2$-bridge knot of type $(p, q)$.
Then the parameter $p/q$ can be expanded into a unique continued fraction of the form
$$[2a_1, 2b_1, \ldots, 2a_n, 2b_n] = 2a_1 + (2b_1 + (\cdots +(2a_n + (2b_n)^{-1})^{-1}\cdots)^{-1},$$ 
where $a_i$ and $b_i$ are non-zero integers for $i \in \{1, 2, \ldots, n\}$ (we follow the convention in \cite{MoS}).
A $2$-bridge knot $K$ is then put in the form shown in Figure \ref{fig:2bridge}, in which each circle indicates a block of some nonzero number of half-twists. Each of the blocks in the middle row has $2a_i$ half-twists, and each in the top row has $2b_i$ half-twists.
We assume that $a_i$ is positive for right-hand twists, and $b_i$ is positive for left-hand twists (for example, see Figure \ref{fig:2bridge_example}).
The presentation in Figure \ref{fig:2bridge} is called the {\it Conway presentation} of  $K_{p/q}$, and the unique continued fraction expansion $[2a_1, 2b_1, \ldots, 2a_n, 2b_n]$ of $p/q$ is called the {\it Conway parameters} of the knot.

\begin{center}
\includegraphics[width=11cm, clip]{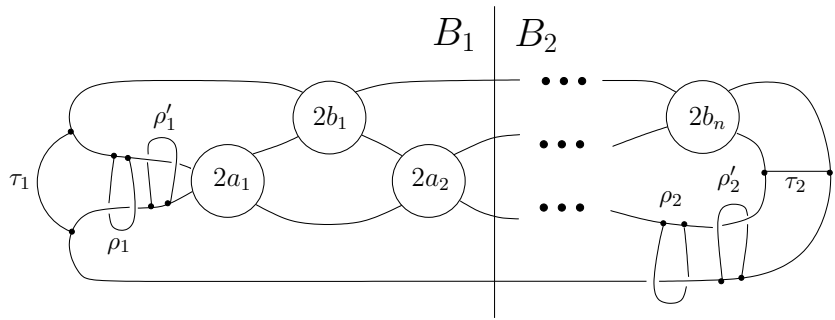}
\captionof{figure}{A Conway presentation of a $2$-bridge knot $K$ with its tunnels.}
\label{fig:2bridge}
\end{center}

Figure \ref{fig:2bridge} also illustrates a $2$-bridge position of $K_{p/q}$ with the two $3$-balls $B_1$ and $B_2$, and the six tunnels.
Then for each  $\rho \in \{\rho_1, \rho'_1, \rho_2, \rho'_2\}$, the solid torus $V$ in the $\rho$-position can be identified with $\Nbd(\rho \cup \alpha)$ where $\alpha$ is one of the arcs of $K - \rho$ that lies in the $3$-ball containing $\rho$.

\begin{theorem}
Let $K = K_{p/q}$ be a $2$-bridge knot with the Conway presentation as in Figure \ref{fig:2bridge}, and let $[2a_1, 2b_1, \ldots, 2a_n, 2b_n]$ be the {\it Conway parameters} of $K$.
Then
\begin{enumerate}
\item the word $m s^{2b_n} l^{-a_n}\cdots s^{2b_1} l^{-a_1}$ is a braid description of the $\rho_1$-position and the $\rho_1'$-position, and
\item the word $m s^{-2a_1} l^{b_1}\cdots s^{-2a_n} l^{b_n}$ is a braid description of the $\rho_2$-position and the $\rho_2'$-position.
\end{enumerate}
\label{thm:2bridgeword}
\end{theorem}

\begin{proof}
Consider the $\rho_1$-position first.
Then the Conway presentation of $K$ as in Figure \ref{fig:2bridge} can be put into a braid position as described as in Figure \ref{fig:2bridge_example}.
The boundary of the solid torus containing $\rho_1$ in the figure is actually the level torus $T \times \{1\}$ in $N = T \times [0, 1]$ in our setup in Section \ref{sec:braid_group}.
Then the braid description of the $\rho_1$-position can be read off directly.
For the $\rho_2$-position, flip the bottom strand to the top in figure \ref{fig:2bridge} to get the Conway parameters $[-2b_n, -2a_n, ...,-2b_1, -2a_1]$.
Then the same interpretation to the case of $\rho_1$-position can be applied.
(In fact, if $p/q' = [-2b_n, -2a_n, ...,-2b_1, -2a_1]$, then $qq'=\pm 1$,$\mod p$, and $K_{p/q}$ and $K_{p/q'}$ are equivalent.)

\begin{center}
\includegraphics[width=8cm, clip]{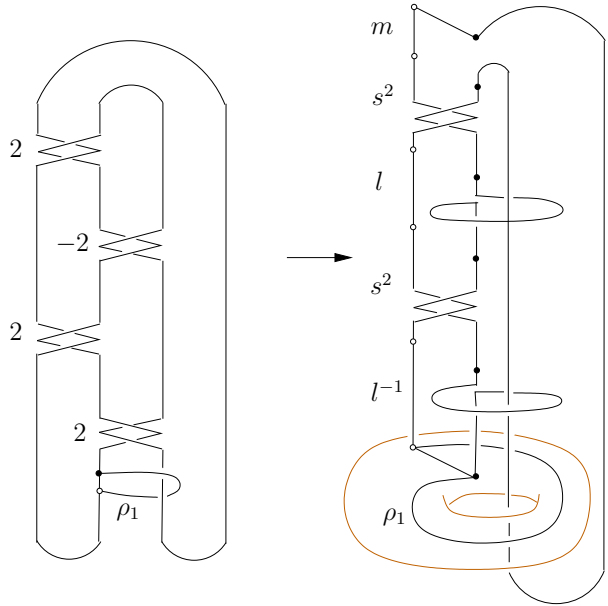}
\captionof{figure}{The $2$-bridge knot with Conway parameters $[2, 2, -2, 2]$. For the Conway presentation of any $2$-bridge knot with parameters $[2a_1, 2b_1, \ldots, 2a_n, 2b_n]$, the braid description of the $\rho_1$-position starts with a single $m$ and the remaining part is a composition of only powers of $s$ and $l$.}
\label{fig:2bridge_example}
\end{center}

For the $\rho'_1$-position, as illustrated in Figure \ref{fig:2bridge_example2}, there is an isotopy sending $K \cup \rho'_1$ to $K \cup \rho_1$ such that the resulting Conway presentation is exactly same to the original one.
Thus we have the same interpretation of the $\rho'_1$-position as the case of the $\rho_1$-position.
The case of the $\rho'_2$-position can be done in the same way.
\end{proof}

\begin{center}
\includegraphics[width=12.4cm, clip]{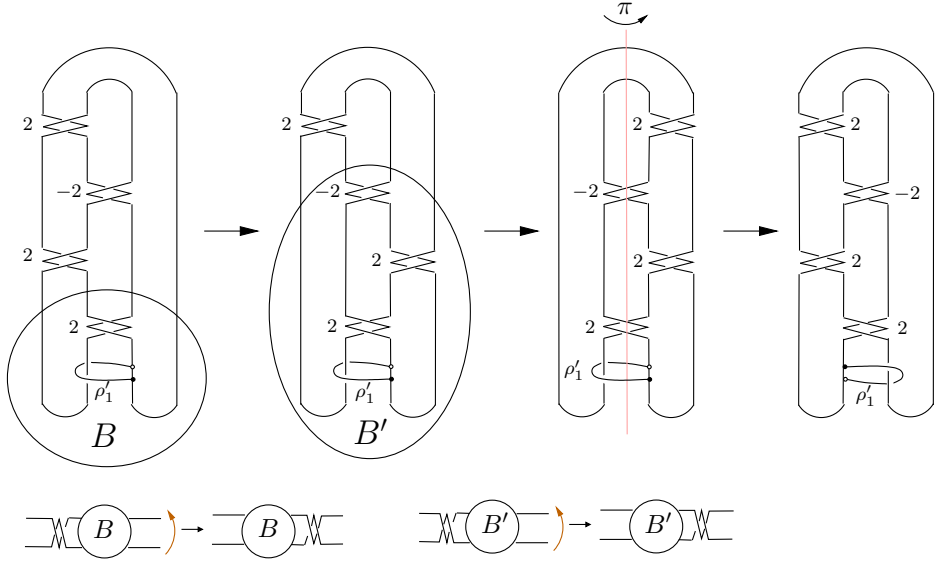}
\captionof{figure}{The first two steps describe the rotations of the balls $B$ and $B'$ with strands, sending the left twists to the right. The final step is the $\pi$-rotation about the vertical axis.}
\label{fig:2bridge_example2}
\end{center}

We note that the $\rho_1$-position and $\rho_1'$-position (and the $\rho_2$-position and $\rho_2'$-position) are isotopic to each other as observed in Figure \ref{fig:2bridge_example2}.
That is, there is an isotopy of the $3$-sphere sending the solid torus containing $\rho_1$ of the $\rho_1$-position to the the solid torus containing $\rho'_1$ of the $\rho_1'$-position.
Thus the $(1, 1)$-lengths (and so the level numbers) of the $\rho_1$-position and $\rho_1'$-position (and the $\rho_2$-position and $\rho_2'$-position) are the same.

\begin{example}
The right-handed trefoil knot $K_{3/2}$ has the Conway parameters $[-2,2]$.
Thus we have that the braid description of the $\rho_1$-position of $K_{3/2}$ is $m s^{2} l $.
We already know the level number of $K_{3/2}$ is $1$ since it is a torus knot, but it can be also verified by considering
$$m s^{2} l \sim s^{-1} m^{-1} l^{-1} s^{-1} \sim m^{-1} l^{-1} \sim l^{-1} m^{-1} l^{-1} m^{-2},$$
and $l^{-1} m^{-1} l^{-1} m^{-2}$ is a torus word.
\end{example}

The rational parameter $p/q$, where $q$ is even and $0 < |q| < p$, admit more general continued fraction expansions of the form $[2c_1, d_1, \ldots, 2c_k, d_k]$ and $[c'_1, 2d'_1, \ldots, c'_{k'}, 2d'_{k'}]$, where $c_j$, $d_j$, $c'_{j'}$ and $d'_{j'}$ are nonzero integers.
(The expansion of each of these forms is not unique in general.)
The $2$-bridge knot $K = K_{p/q}$ can be put in the form in Figure \ref{fig:2bridge} with these parameters.
That is, change the number of half-twists $2a_j$ to $2c_j$ (or $c'_{j'}$) and $2b_j$ to $d_j$ (or $2d'_{j'}$).
Then, the braid descriptions of the $(1, 1)$-positions of $K$ can be obtained from these ``non-Conway parameters'' by the same way as in the proof of Theorem \ref{thm:2bridgeword}.
Hence we have a more general version of Theorem \ref{thm:2bridgeword} as follows.

\begin{theorem}
Let $K = K_{p/q}$ be a $2$-bridge knot, and let $[2c_1, d_1, \ldots, 2c_k, d_k]$ and $[c'_1, 2d'_1, \ldots, c'_{k'}, 2d'_{k'}]$ be the expansions of $p/q$ as above.
Then,
\begin{enumerate}
\item the word $m s^{d_k} l^{-c_k}\cdots s^{d_1} l^{-c_1}$ is a braid description of the $\rho_1$-position and the $\rho_1'$-position, and
\item the word $m s^{-c'_1} l^{d'_1}\cdots s^{-c'_{k'}} l^{d'_{k'}}$ is a braid description of the $\rho_2$-position and the $\rho_2'$-position.
\end{enumerate}
\label{thm:2bridgeword2}
\end{theorem}

Now we can give an upper bound of the level number of each $(1, 1)$-position of $2$-bridge knot.

\begin{theorem}
Let $K = K_{p/q}$ be a $2$-bridge knot, and let $[2c_1, d_1, \ldots, 2c_k, d_k]$ and $[c'_1, 2d'_1, \ldots, c'_{k'}, 2d'_{k'}]$ be the expansions of $p/q$ as above.
Then,
\begin{enumerate}
\item an upper bound of the level number of the $\rho_1$-position and the $\rho_1'$-position is $|\alpha_2|+\cdots+|\alpha_k|+2$, where $\alpha_i=\min\{|c_i|,2\}$, and
\item an upper bound of the level number of the $\rho_2$-position and the $\rho_2'$-position is $|\beta_2|+\cdots+|\beta_k|+2$, where $\beta_i=\min\{|d_i'|,2\}$.
\end{enumerate}
\label{thm:lenBraid}
\end{theorem}

\begin{proof}
By Theorem \ref{thm:2bridgeword2}, the word $ms^{d_k}l^{-c_k}\cdots s^{d_1}l^{-c_1}$ is a braid description of the $\rho_1$-position and the $\rho_1'$-position.
If $c_i\neq \pm1$, we have
$$s^{d_i} l^{-c_i} \sim s^{d_i+1} msml^{-c_i}$$
with $m$ and $ml^{-c_i}$ torus words.
For each $i \in \{2, \ldots, k\}$ with $c_i\neq \pm1$, we replace each $s^{d_i} l^{-c_i}$ with such an expression.
The final $s^{d_1}l^{-c_1}$ can be rewritten as $s^{d_1}l^{-c_1}m^{-1}$, giving the stated bound.
The proof for the second case is similar.
\end{proof}

\begin{example}
The knot $K_{13/4}$ has the Conway parameters $[4,-2,2,-2]$.
Thus the $\rho_1$-position has the braid description $m s^{-2} l^{-1} s^{-2} l^{-2}$, and the $\rho_2$-position has the braid description $m  s^{-4} l^{-1} s^{-2} l^{-1}$.
For the $\rho_1$-position,
$$m s^{-2} l^{-1} s^{-2} l^{-2} \sim m s^{-2} l^{-1} s^{-2} l^{-2} m^{-1}.$$
Thus the level number of the $\rho_1$-position is at most $3$.
Consider the $\rho_2$-position. We have
$$m s^{-4} l^{-1} s^{-2} l^{-1} \sim m s^{-4} s l^{2} s \sim m s^{-3} l^{2} \sim l^{-1} m s^{-3} l^{2} m^{-1}.$$
Since $l^{-1} m$ and $l^{2} m^{-1}$ are torus words, the level number of the $\rho_2$-position is at most $2$.
Since $K_{13/4}$ is not a torus knot, the level number of $K_{13/4}$ is $2$.
\end{example}

\begin{example}
In many cases, non-Conway parameters' provide sharper upper bounds of the level number than the Conway parameters. For example, consider the Conway parameters $[6,2,2,-2,2,-2,2,-2,2,-2,2,2,6,4]$.
By Theorem \ref{thm:lenBraid}, we have the upper bound of the level number of the knot with the parameters as
$\min \{|1|,2\}+\min \{|1|,2\}+ \min \{|1|,2\}+\min \{|1|,2\}+\min \{|1|,2\}+\min \{|3|,2\}+ 2 = 9$.
On the other hand, we observe that
$[6,2,2,-2,2,-2,2,-2,2,-2,2,2,6,4] = [6, 3, -10, 3, 6, 4]$,
and hence we have a smaller upper bound
$\min \{|-5|,2\}+\min \{|3|,2\}+ 2 = 6$.
\end{example}

\begin{example}
Let $K$ be a $2$-bridge knot with Conway parameters of the form $\pm[2,-2,...,2,-2]$ of length $2n$.
Then the word $\omega = m s^{-2}  l^{-1} \cdots s^{-2}  l^{-1}$ is a braid description of (any) $(1, 1)$-position of $K$.
Considering
$$\omega \sim m^{-1} l^n s^{-1} l^{-1}\sim m^{-1} l^n l s \sim m^{-1} l^{n+1}\sim l^{n} m^{-1} l^{n+1} m^{-1},$$
the level number of $K$ with this parameters is $1$ since $l^{n} m^{-1} l^{n+1} m^{-1}$ is a torus word.
For the parameters $-[2,-2,...,2,-2]$, we found a braid description that is $(1, 1)$-equivalent to the torus word $l^{-n-1} m^{-1} l^{-n} m^{-1}$ in the same way.
Thus $K$ must be a torus knot. In fact, since $\pm [2,-2,...,2,-2] = \pm (2n+1)/2n$, $K = K_{\pm (2n+1)/2n}$, which is the $(\pm 2n+1,2)$-torus knot.
\end{example}

\begin{example}
Let $K$ be a $2$-bridge knot with Conway parameters of the forms
\begin{itemize}
\item $[2a, 2b]$,
\item $\pm[2a, 2, \ldots,2 ,-2]$ or $\pm[2a,-2, \ldots,2 ,-2]$ (of length $2n$), and
\item $\pm[2, -2, \ldots,2, -2, 2, 2b]$ or $\pm[2, -2, \ldots, 2, -2, -2, 2b]$ (of length $2n$).
\end{itemize}

We verify that the level number of the $2$-bridge knots with those parameters is at most $2$.
First, for the parameters $[2a, 2b]$, the braid description of the $\rho_1$-position is
$$m s^{2b}  l^{-a}\sim m s^{2b}  l^{-a} m^{-1},$$
and similarly, the braid description of the $\rho_2$-position is
$$m s^{-2a}  l^{b}\sim m s^{-2a}  l^{b} m^{-1}.$$
Thus the level number of $K$ is at most $2$.

For the parameters $[2a, -2, \ldots, 2, -2]$, the braid description of the $\rho_1$-position is
$$m s^{-2}  l^{-1} \ldots s^{-2}  l^{-a} \sim m  s^{-1}  l^{n-1} s^{-1} l^{-a}\sim m  l s l  l^{n-1}  s^{-1} l^{-a}\sim l s^2 m  s  l^n  s^{-1} l^{-a}$$
$$\sim s^{-1} m^{-1} l^n  s^{-1} l^{-a} \sim m^{-1} l^n  s^{-1} l^{-a} \sim l^{n+1} m^{-1} l^n m^{-1} s^{0} m s^{-1} l^{-a}m^{-1}.$$
Thus the level number of the $\rho_1$-position is at most $3$.
On the other hand, the braid description of the $\rho_2$-position is
$$m s^{-2a+1} l^{n} \sim l^{-1} m s^{-2a+1} l^{n} m^{-1}.$$
Thus the level number of the $\rho_2$-position is at most $2$, so is that of $K$.

Similarly, for the parameters $-[2a, -2, \ldots, 2, -2]$, we have the braid description of $\rho_1$-position is
$$l^{-n} m^{-1} l^{-(n-1)} m^{-1} s^{0} m s^{1} l^{-a}m^{-1},$$
and so the level number of the $\rho_1$-position is at most $3$.
The braid description of $\rho_2$ is
$$m s^{-2a-1} l^{-n} \sim l^{-1} m s^{-2a-1} l^{-n} m^{-1},$$
so the level number is at most $2$.
The calculations for the parameters $\pm[2, -2, \ldots,2, -2, 2, 2b]$ and  $\pm[2, -2, \ldots, 2, -2, -2, 2b]$ are similar.
We conjecture that every $2$-bridge knot of level number at most two has Conway parameters one of the forms above.
\end{example}

\section{The $(-2,3,7)$-pretzel knot}
\label{sec:237}

The tunnels of the $(-2,3,7)$-pretzel knot, denoted by $K$, were
determined by D.~Heath and H.~Song~\cite{HS}.
There are exactly four mutually non-isotopic unknotting tunnels for $K$.
Figure~\ref{fig:four_tunnels} shows $K$ along with its four tunnels.
It can be verified directly that they are all $(1, 1)$-tunnels.
Actually, the tunnel $\tau_1$ cuts off $K$ into two arcs, and a regular neighborhood of the union of $\tau_1$ and one of the two arcs is a standard solid torus, say $V$.
Then $\tau_2$ is contained in the exterior of $V$, and hence $\tau_1$ and $\tau_2$ turn out to be $(1, 1)$-tunnels with respect to the standard torus $\partial V$ which gives a $(1, 1)$-position of $K$.
We call the $(1, 1)$-position determined by $\tau_1$ (together with $\tau_2$) simply the {\it $\tau_1$-position} of $K$.
Similarly, the tunnel $\tau_3$ (together with $\tau_4$) determine the {\it $\tau_3$-position} of $K$.
In this section, we show that the level numbers of the two $(1, 1)$-positions are two and obtain their braid descriptions.

\begin{center}
\includegraphics[width=7cm, clip]{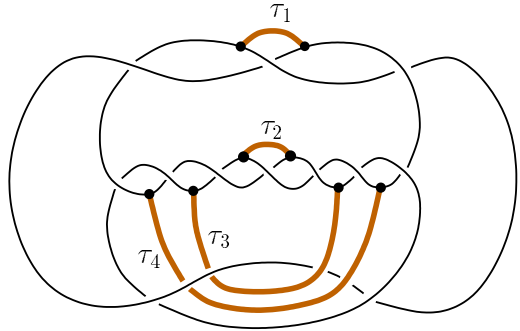}
\captionof{figure}{The four tunnels of the $(-2,3,7)$-pretzel knot $K$.}
\label{fig:four_tunnels}
\end{center}

A leveling of the $\tau_1$-position is indicated in Figure~\ref{fig:first_level_position}.
The isotopy shown in the sequence of drawings moves finally the knot to reside on two level tori, connected by two strands with a single half twist. The $\tau_1$-position is a slight perturbation
of this leveling of level number $2$.
Since $K$ is not a torus knot, the level number of the $\tau_1$-position is $2$.
From Figure~\ref{fig:first_level_position}, we notice that the portion of $K$ in the outer level torus determines the torus word of type $(0, -1)$, and that in the inner level torus determines the torus word of type $(-2, 7)$.
Thus a braid description of the $\tau_1$-position can be written and simplified as
$$m^{-1}s^{-1}ml^{-1}m^3l^{-1}m^3 ~\sim~ m^{-1}s^{-1}ml^{-1}m^3l^{-1} ~\sim~ m^2l^{-1}m^3l^{-1}.$$

\begin{center}
\includegraphics[width=12cm, clip]{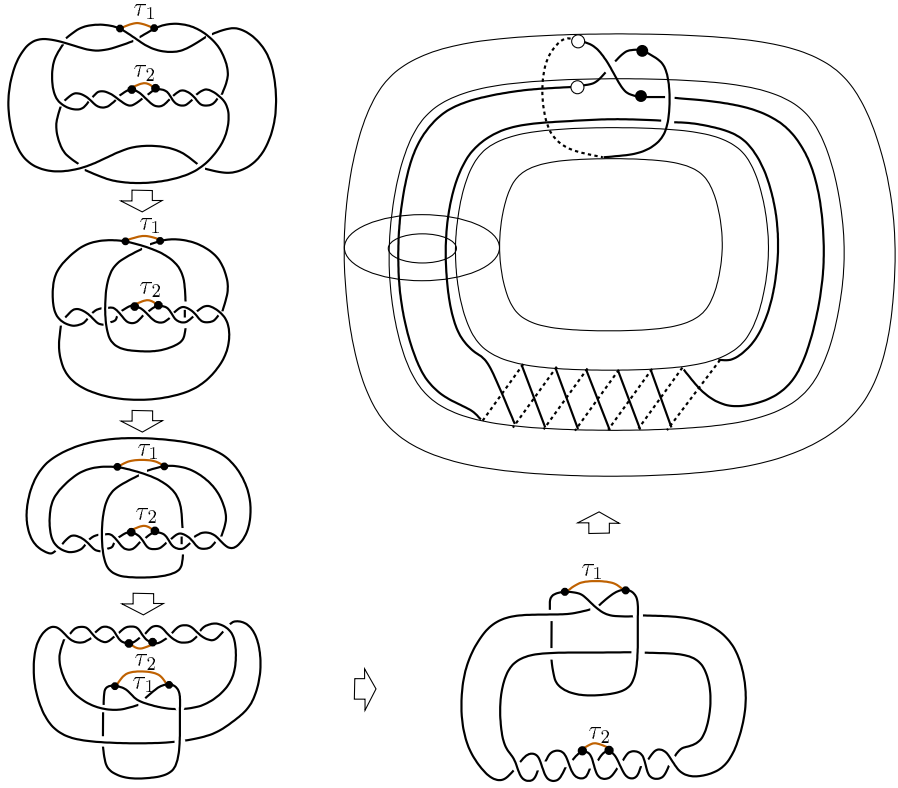}
\captionof{figure}{A leveling of $\tau_1$-position.}
\label{fig:first_level_position}
\end{center}

We turn now to $\tau_3$ and $\tau_4$. Similarly to
Figure~\ref{fig:first_level_position}, Figure~\ref{fig:tau3tau4} shows an isotopy that moves $K$ onto two level tori, connected by a single half-twist.
Again, the level number of the $\tau_3$-position is $2$.
The portion of $K$ in the outer level torus determines the torus word of type $(-1, 2)$, and that in the inner level torus determines the torus word of type $(2, -3)$.
Thus a braid description of the $\tau_3$-position can be written and simplified as

$$m l^{-1} m  s^{-1} m^{-1} l m^{-1} l m^{-1} ~\sim~ m l^{-1} m  s^{-1} m^{-1} l m^{-1} l.$$

We summarize the above verification as follows.

\begin{theorem}
The $(-2,3,7)$-pretzel admits exactly two $(1, 1)$-positions up to isotopy, the $\tau_1$-position and the $\tau_3$-position.
Their level number is $2$, and
\begin{itemize}
\item the $\tau_1$-position has the braid description $m^2l^{-1}m^3l^{-1}$, and
\item the $\tau_3$-position has the braid description
    $m l^{-1} m  s^{-1} m^{-1} l m^{-1} l$.
\end{itemize}
\end{theorem}

\begin{center}
\includegraphics[width=12cm, clip]{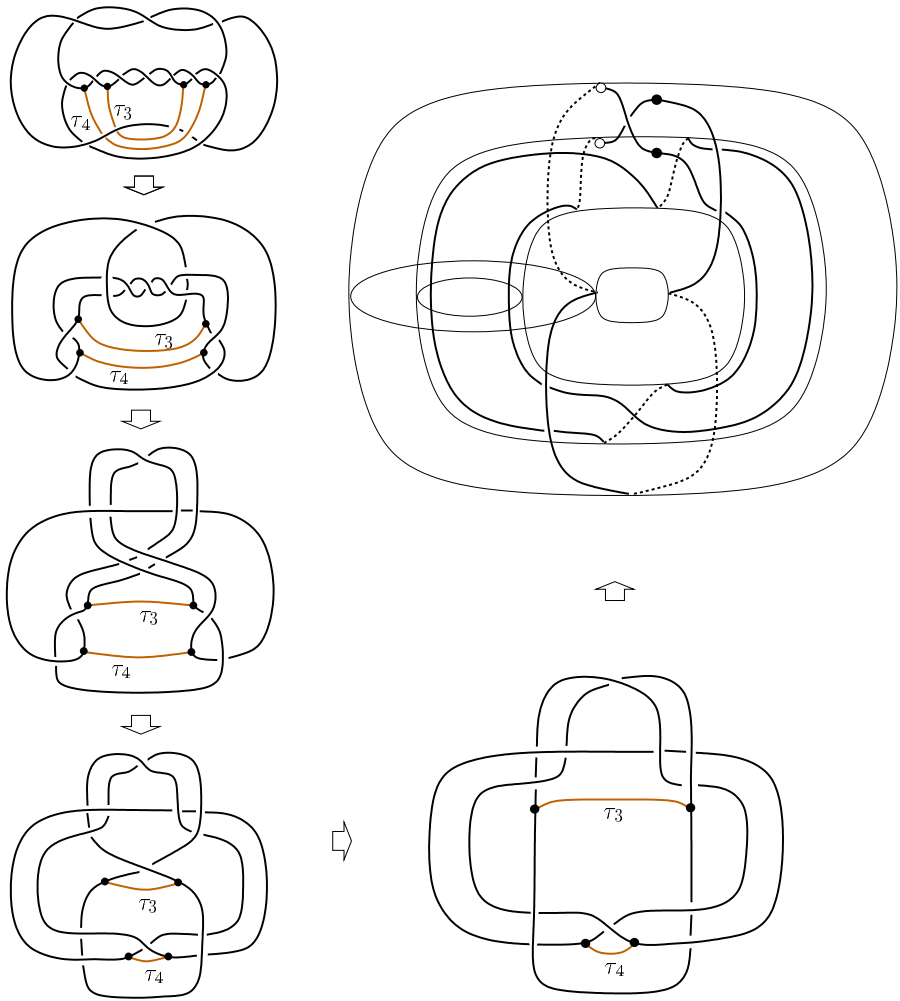}
\captionof{figure}{A leveling of $\tau_3$-position.}
\label{fig:tau3tau4}
\end{center}

\bibliographystyle{amsplain}

\end{document}